   \documentclass[twoside,reqno,11pt]{fcaa-var} %

\usepackage{graphicx}
\usepackage{epsfig}

\usepackage{amsthm}
\usepackage{amsmath}
\usepackage{latexsym}
\usepackage{amsfonts}
\usepackage{amssymb}

\newtheoremstyle{theorem}
  {15pt}          
  {15pt}  
  {\sl}  
  {\parindent}
  {\sc}  
  {. }   
  { }    
  {}     
\theoremstyle{theorem}

\newtheorem{theorem}{Theorem}[section]
\newtheorem{corollary}{Corollary}[section]
\newtheorem{proposition}{Proposition}[section]

\newtheoremstyle{defi}
  {15pt}          
  {15pt}  
  {\rm}  
  {\parindent}     
  {\sc}  
  {. }    
  { }    
  {}     
\theoremstyle{defi}
\newtheorem{definition}{Definition}[section]
\newtheorem{remark}{Remark}[section]

 \def\proofname{\indent\rm P\ r\ o\ o\ f.}
 
 \def\qed{\hfill$\Box$} 
   \usepackage{hyperref} 
  \def\theequation{\arabic{section}.\arabic{equation}}

 \title[Fractional Dirichlet distribution]{A fractional generalization of the Dirichlet distribution and related distributions}
 \author[\normalsize E. Di Nardo, F. Polito, E. Scalas]{\normalsize Elvira Di Nardo$^1$, Federico Polito$^1$, Enrico Scalas$^2$}

 \usepackage{color,bm}
 \allowdisplaybreaks
 
\begin{document}

 \vbox to 1.5cm { \vfill }


 \bigskip \medskip

 \begin{abstract}

This paper is devoted to a fractional generalization of the Dirichlet distribution. The form of the multivariate distribution is derived assuming that the $n$ partitions of the interval $[0,W_n]$ are independent and identically distributed random variables following the generalized Mittag-Leffler distribution. The expected value and variance of the one-dimensional marginal are derived as well as the form of its probability density function. A related generalized Dirichlet distribution is studied that provides a reasonable approximation for some values of the parameters. The relation between this distribution and other generalizations of the Dirichlet distribution is discussed. Monte Carlo simulations of the one-dimensional marginals for both distributions are presented.

 \medskip


 \smallskip


 \end{abstract}

 \maketitle

 \vspace*{-20pt}


	\section{Introduction}
	
	\noindent Let us consider a finite sequence of $n$ positive random variables $Z_1, \ldots, Z_n$. For instance, these variables can represent the wealth of $n$ economic agents if indebtedness is not allowed. Let us denote the sum as
	$W_n = Z_1 + \ldots +Z_n$. In the wealth interpretation this is the total wealth. If we define the wealth fraction of the $i$-th agent as $Q_i = Z_i/W_n$, we get a partition of the interval $[0,1]$ represented by the sequence $\mathbf{Q} = (Q_1, \ldots, Q_n)$ such that $Q_1 + \ldots + Q_n = 1$ almost surely. We are particularly interested in multivariate distributions for the sequence $\mathbf{Q}$ whose one-dimensional marginals have heavy tails.  If we further assume that the random variables $Z_1, \ldots, Z_n$ are independent and identically distributed, there is a nice and immediate relationship with point processes of renewal type. In this case, the variables $Z_i$ can be interpreted as inter-event intervals and the partial sums $W_k = \sum_{i=1}^k Z_i$, with $k \leq n$, are the epochs of the events.
	
	In order to clarify the relationship, we start by recalling some basic facts on the time-fractional Poisson process as we are going to use and generalize it in the next section. From \cite{MR2535014,
		MR2120631} we know that the time-fractional Poisson process $N^\nu = (N^\nu(t))_{t \ge 0}$, $\nu \in (0,1]$, can be defined as a renewal
		process with independent and identically distributed inter-event waiting times $\mathcal{T}_j$, $j \in \mathbb{N}^* = \{1,2,\dots\}$, with probability density function
		\begin{align}
			\label{part}
			\mathbb{P}(\mathcal{T}_j \in \mathrm dt) = \lambda t^{\nu-1} E_{\nu,\nu} (-\lambda t^\nu) \, \mathrm dt, \qquad
			\lambda > 0, \: t > 0,
		\end{align}
		where
		\begin{align}
			E_{\alpha,\beta}(z) = \sum_{r=0}^\infty \frac{z^r}{\Gamma(\alpha r + \beta)}, \qquad z,\alpha,\beta \in \mathbb{C}, \: \Re(\alpha)>0,
		\end{align}
		is the two-parameter Mittag--Leffler function. Note that for $\nu=1$, the waiting times $\mathcal{T}_j$ are exponentially
		distributed and $N^1$ is the homogeneous Poisson process. Moreover, the Laplace transform
		of the probability density \eqref{part} takes a very compact form. Indeed, we have
		\begin{align}
			\int_0^\infty e^{-z t} \mathbb{P}(\mathcal{T}_j \in \mathrm dt) = \frac{\lambda}{\lambda+z^\nu}, \qquad z > 0.
		\end{align}
		Let us now indicate with $T_k$, $k \in \mathbb{N}^*$, the random occurrence time of the $k$-th event of the stream of events defining $N^\nu$. From the renewal structure
		of $N^\nu$ we readily obtain that the Laplace transform of $T_k$ reads
		\begin{align}
			\int_0^\infty e^{-z t} \mathbb{P}(T_k \in \mathrm dt) = \left( \frac{\lambda}{\lambda+z^\nu} \right)^k, \qquad z > 0,
		\end{align}
		which in turn corresponds to the Laplace transform of a function involving the three-parameter Mittag--Leffler function (also known as the Prabhakar function -- see \cite{giusti}).
		In particular, the three-parameter Mittag--Leffler function is defined as
		\begin{align}
			E_{\alpha,\beta}^\delta(z) = \sum_{r=0}^\infty \frac{z^r}{\Gamma(\alpha r + \beta)}
			\frac{\Gamma(\delta+r)}{r! \Gamma(\delta)}, \qquad z,\alpha,\beta,\delta \in \mathbb{C}, \: \Re(\alpha)>0,
		\end{align}
		and we know by direct calculation that (see e.g.\ \cite{mathai}, formula (2.3.24))
		\begin{align}
			\label{20de}
			\int_0^\infty t^{\beta-1} e^{-zt} E_{\alpha, \beta}^\delta (\zeta t^\alpha)\, \mathrm dt =
			z^{-\beta} \left( 1-\zeta z^{-\beta} \right)^{-\delta},
		\end{align}
		where $\Re(\alpha)>0$, $\Re(\beta)>0$, $\Re(z)>0$, $z > |\zeta|^{1/\Re(\alpha)}$. Using \eqref{20de},
		we obtain
		\begin{align}
			\mathbb{P}(T_k \in \mathrm dt) = \lambda^k t^{\nu k-1} E_{\nu,\nu k}^k (-\lambda t^\nu) \, \mathrm dt, \qquad
			\lambda>0, \: t > 0, \: \nu \in (0,1], \: k \in \mathbb{N}^*.
		\end{align}
		\begin{remark}\label{remErlang}
    		Note that, for $\nu=1$, the above density reduces to that of an $\text{Erlang}(\lambda,k)$ distributed random variable. This can be
    		seen by simply noticing that
    		\begin{align}
    			\mathbb{P} (T_k \in \mathrm dt) = \mathrm dt \, \lambda^k t^{k-1} \sum_{r=0}^\infty \frac{(-\lambda t)^r \Gamma(r+k)}{\Gamma(r+k)\Gamma(k)r!}
    			= \mathrm dt \frac{\lambda^k t^{k-1}e^{-\lambda t}}{(k-1)!}, \qquad k \in \mathbb{N}^*.
    		\end{align}
		\end{remark}
		\begin{remark}\label{remDirichlet}
		The $\text{Erlang}(\lambda,k)$ distribution is a special case of the $\text{Gamma}(a,c)$ distribution. Consider a sequence $Z_1, \ldots, Z_n$ of independent random variables each following a $\text{Gamma}$ distribution of parameter $(a_1,c), \dots, (a_n,c)$. It is well known that their sum $W_n$ is still a $\text{Gamma}$ of parameter $(a_1 + \ldots + a_n,c)$. Then the sequence of fractions $Q_1, \ldots, Q_n$ has a joint $(N-1)$-dimensional Dirichlet distribution of parameters $a_1, \ldots, a_n$ with density
		\begin{equation}
		    \label{classicDirichlet}
		    f_{\bf Q} (q_1, \ldots, q_{n-1}) = \frac{\Gamma(a_1 + \ldots + a_n)}{\Gamma(a_1)\cdot \ldots \cdot \Gamma(a_n)} q_1^{a_1-1}\cdots \left(1-\sum_{i=1}^{n-1} q_i \right)^{a_n-1}
		\end{equation}
		with $q_1 + \dots + q_n = 1$ and is independent of $W_n$.
		\end{remark}
        The proof of the results in Remark \ref{remDirichlet} can be found in several textbooks and lecture notes (see e.g. \cite{Bertoin}, Lemma 1.5).
        \begin{remark}
            The random variables $\mathcal{T}_j$ have the following asymptotic behaviour for $t \to \infty$ \cite{gorenflo}:
            \begin{equation}
                \mathbb{P}(\mathcal{T}_1 > t) \sim \frac{\sin(\nu \pi)}{\pi} \frac{\Gamma(\nu)}{t^\nu}, \,\, t \gg 1;
            \end{equation}
            therefore, their sums $T_k$ belong to the basin of attraction of the $\nu$-stable subordinator.
        \end{remark}	
        \begin{remark}
            The distributions considered in the present paper belong to the class of distributions on the simplex discussed in \cite{fav} (see \eqref{lap} below and \cite{bondesson}).
        \end{remark}

        This paper contains the following material. Section 2 concerns the definition and properties of the fractional Dirichlet distribution. Section 3 mirrors section 2 and is devoted to the generalized Dirichlet distribution. Section 4 explains how to simulate the fractional Dirichlet distribution and  presents the results of Monte Carlo simulations in order to illustrate the relation between the fractional Dirichlet distribution and the generalized Dirichlet distribution.

	\section{Construction of the fractional Dirichlet distribution}\label{sec}
	
        \noindent Based on Remark \ref{remErlang} and Remark \ref{remDirichlet}, we now define a generalization of the Gamma distribution and we immediately present a fractional generalization of the Dirichlet distribution.
		\begin{definition}[Fractional Gamma distribution]
		    Let $X$ be a positive real valued random variable with
		    distribution
		    \begin{align}\label{miche}
		        \mu(\mathrm dx)
		        = \mathbb{P}(X \in \mathrm dx) = \lambda^\beta x^{\nu \beta-1} E_{\nu,\nu\beta}^\beta (-\lambda x^\nu) \, \mathrm dx,
		    \end{align}
		    where $\lambda > 0$, $x>0$, $\beta >0$, $\nu \in (0,1]$. Then $X$ is said to be distributed as a fractional Gamma of parameters $\lambda,\beta,\nu$ (we write $X \sim FG(\lambda, \beta, \nu)$) (see \cite{pillai}; for applications to renewal processes see \cite{cahoypolito, mich1, mich2, mich3}).
		\end{definition}
		\begin{remark}
    		The Laplace transform of $\mu$ reads
		    \begin{align}
		    	\label{lap}
		    	\int_0^\infty e^{-zx} \mu(\mathrm dx) = \left( \frac{\lambda}{\lambda+z^\nu} \right)^\beta, \qquad z > 0. 
		    \end{align}
		\end{remark}
		By means of \eqref{miche}, we will construct a generalization of the Dirichlet distribution.
		We consider $n$ independent random variables $Z_i$, $i=1,\dots,n$, distributed
		as fractional Gamma random variables of parameters $(1,\beta_i,\nu)$, $\nu \in (0,1]$, $\beta_i>0$, $i=1,\dots, n$, respectively.
        Furthermore, define the sum
		$W = Z_1+\dots +Z_n$, set $Q_i=Z_i/W$, $i=1,\dots,n$, and consider
		the transformation
		\begin{align}
			(Z_1,\dots,Z_n) \longrightarrow \biggl(W Q_1, \dots, W Q_{n-1}, W\Bigl(1-\sum_{i=1}^{n-1} Q_i\Bigr)\biggr).
		\end{align}
		Note that, from \eqref{lap}, the distribution of $W$
		is fractional Gamma as well, i.e.\ $W \sim FG(1,\bar{\beta},\nu)$, where $\bar{\beta} = \sum_{i=1}^n \beta_i$.
		The joint pdf of the vector $(W,\bm{Q})=(W,Q_1, \dots, Q_{n-1})$ reads
		\begin{align}
			f_{(W,\bm{Q})} & (y,q_1,\dots,q_{n-1}) \\
			= {} & \left[ \prod_{i=1}^{n-1} (y q_i)^{\nu \beta_i-1} E_{\nu,\nu \beta_i}^{\beta_i} (-(yq_i)^\nu) \right] 
		    \left[ y(1-\sum_{i=1}^{n-1}q_i) \right]^{\nu \beta_n-1} \notag\\
			& \times E_{\nu,\nu\beta_n}^{\beta_n} \biggl( -\Bigl(y-y\sum_{i=1}^{n-1}q_i\Bigr)^\nu \biggr) y^{n-1} \notag \\
			= {} & \left( \prod_{i=1}^{n-1} y^{\nu \beta_i -1} \right) y^{\nu \beta_n -1} y^{n-1}
			\left[ \prod_{i=1}^{n-1} q_i^{\nu \beta_i-1} E_{\nu,\nu \beta_i}^{\beta_i}(-(yq_i)^\nu) \right] \notag\\
			& \times\left( 1-\sum_{i=1}^{n-1} q_i \right)^{\nu \beta_n-1} \hspace{-.5cm} E_{\nu,\nu \beta_n}^{\beta_n}
			\biggl( -\Bigl(y-y\sum_{i=1}^{n-1} q_i\Bigr)^\nu \biggr) \notag \\
			= {} & y^{\nu \bar{\beta} -1} \left[ \prod_{i=1}^{n-1} q_i^{\nu \beta_i-1} E_{\nu,\nu \beta_i}^{\beta_i}
			(-(yq_i)^\nu) \right] \left( 1-\sum_{i=1}^{n-1} q_i \right)^{\nu \beta_n-1} \notag \\
			& \times E_{\nu,\nu \beta_n}^{\beta_n} \biggl( -\Bigl( y-y\sum_{i=1}^{n-1}q_i \Bigr)^\nu \biggr). \notag
		\end{align}
		The joint probability density function of $\bm{Q} = (Q_1, \dots, Q_{n-1})$ is then obtained by marginalization. Hence,
		\begin{align}\label{integralformula}
			f_{\bm{Q}}&(q_1,\dots,q_{n-1}) = \left( \prod_{i=1}^{n-1} q_i^{\nu \beta_i-1} \right)
			\left( 1-\sum_{i=1}^{n-1} q_i\right)^{\nu \beta_n-1} \\
			& \times \int_0^\infty y^{\nu \bar{\beta} -1} \prod_{i=1}^{n-1} E_{\nu,\nu \beta_i}^{\beta_i}
			(-(yq_i)^\nu) E_{\nu,\nu \beta_n}^{\beta_n} \biggl( -\Bigl( y-y\sum_{i=1}^{n-1}q_i \Bigr)^\nu \biggr) \mathrm dy. \notag
		\end{align}

        \begin{remark}		
        	On the $n$-dimensional simplex $\Delta_n$ the probability density of the random vector $(Q_1,\dots,Q_n)$, where $\sum_n Q_n = 1$ a.s., writes		
    		\begin{align}\label{cucu}
    			& \mathbb{P}\left( (Q_1, \dots, Q_n) \in \mathrm d (q_1, \dots, q_n) \right) \\
	    		& = \prod_{i=1}^n q_i^{\nu \beta_i-1} \int_0^\infty y^{\nu \bar{\beta} -1} \prod_{i=1}^n E_{\nu,\nu \beta_i}^{\beta_i}
	    		\left( -(yq_i)^\nu \right) \, \mathrm dy. \notag
	    	\end{align}
	
    		Notice that for $\nu=1$ the integral in the rhs of  \eqref{cucu} can be easily solved and the Dirichlet$(\bm{\beta} = (\beta_1,\dots,\beta_n))$ is obtained. In this case
	    	$(Q_1,\dots,Q_n)$ is uniformly distributed on $\Delta_n$ for $\beta_i=1$, $i \in \mathbb{N}^*$.
	    	
	    	If $\nu \in (0,1)$ with  $\beta_i=1$, we have
	    	\begin{align}
	    		\mathbb{P}\left( (Q_1, \dots, Q_n) \in \mathrm d (q_1, \dots, q_n) \right)
	    		= \prod_{i=1}^n q_i^{\nu-1} \int_0^\infty y^{\nu n -1}  \prod_{i=1}^n E_{\nu,\nu} (-(yq_i)^\nu) \, \mathrm dy,
		    \end{align}
		    which is symmetric but not uniform.
		    
    		If we let instead $\beta_i=1/\nu$ (again symmetric) we obtain
		    \begin{align}
		    	\mathbb{P}\left( (Q_1, \dots, Q_n) \in \mathrm d (q_1, \dots, q_n) \right)
		    	= \int_0^\infty y^{n-1} \prod_{i=1}^n E_{\nu,1}^{1/\nu} (-(yq_i)^\nu) \, \mathrm dy.
		    \end{align}
		\end{remark}
		
        \subsection{Properties}
        
            The derivation of the marginal moments can be done explicitly using the formulas in Section 2.2 of \cite{fav}.
            
            \begin{proposition}
                Let $\bm{Q} = (Q_1,\dots,Q_{n-1})$ be a random vector distributed with pdf \eqref{integralformula}. For each $j = 1, \dots, n-1$, we have,
                \begin{align}
                    & \mathbb{E} Q_j = \frac{\beta_j}{\bar{\beta}}, \\
                    & \mathbb{V}\text{ar} \, Q_j = \frac{\beta_j(\bar{\beta}-\beta_j)}{\bar{\beta}^2(\bar{\beta}+1)}\left( 1+ \bar{\beta}(1-\nu) \right). \label{va}
                \end{align}
            \end{proposition}
            \begin{proof}
                By Proposition 2 of \cite{fav} we have
                \begin{align}
                    \mathbb{E} Q_j &= - \int_0^\infty \left( \frac{\mathrm{d}}{\mathrm{d}z} \left( \frac{1}{1+z^\nu} \right)^{\beta_j} \right)\left( \frac{1}{1+z^\nu} \right)^{\bar{\beta}-\beta_j}\mathrm{d}z \\
                    & = \int_0^\infty \frac{\beta_j \nu z^{\nu-1}}{1+z^\nu}
                    (1+z^\nu)^{-\bar{\beta}}\mathrm{d}z \notag \\
                    & = \beta_j \int_0^\infty \frac{\mathrm{d}w}{(1+w)^{\bar{\beta+1}}} = \frac{\beta_j}{\bar{\beta}}. \notag
                \end{align}
                Similarly, the second moment writes
                \begin{align}
                     \mathbb{E} Q_j^2 &= \int_0^\infty z\left( \frac{\mathrm{d}^2}{\mathrm{d}z^2} \left( \frac{1}{1+z^\nu} \right)^{\beta_j} \right)\left( \frac{1}{1+z^\nu} \right)^{\bar{\beta}-\beta_j}\mathrm{d}z \\
                     & = \int_0^\infty \left[ \frac{\nu^2 z^{2\nu-2}\beta_j(\beta_j+1)}{(1+z^\nu)^{\beta_j+2}} -\frac{\beta_j \nu (\nu-1) z^{\nu -2}}{(1+z^\nu)^{\beta_j+1}} \right] (1+z^\nu)^{-\bar{\beta}+\beta_j}\mathrm{d}z \notag \\
                     & \!\!\!\!\overset{z^\nu=w}{=} \nu \beta_j (\beta_ +1) \int_0^\infty \frac{w}{(1+w)^{\bar{\beta}+2}}\mathrm{d}w + (1-\nu) \beta_j \int_0^\infty \frac{\mathrm{d}w}{(1+w)^{\bar{\beta}+1}} \notag \\
                     & = \nu\frac{\beta_j(\beta_j + 1)}{\bar{\beta}(\bar{\beta}+1)} + (1-\nu) \frac{\beta_j}{\bar{\beta}}, \notag
                \end{align}
                and hence after some computation
                \begin{align}
                    \mathbb{V}\text{ar} \, Q_j = \frac{\beta_j(\bar{\beta}-\beta_j)}{\bar{\beta}^2(\bar{\beta}+1)}\left( 1+ \bar{\beta}(1-\nu) \right).
                \end{align}
            \end{proof}
            
            \begin{remark}
                Notice that the first factor of the variance \eqref{va} is in fact the variance of a one-dimensional marginal of a Dirichlet$(\bm{\beta})$ distribution. It follows that the
                marginals are overdispersed with respect to those of
                a Dirichlet$(\bm{\beta})$ distribution.
            \end{remark}

            We now proceed by analyzing  the aggregation property and therefore the marginal distributions.

            \begin{proposition}[Aggregation property]
                Consider the distribution defined in equation \eqref{integralformula} and the random variable $\mathcal{Q} = \sum_{j=1}^k Q_{i_j}$ where $1 < k < n$ and $i_j$ denotes any permutation of the indices. Then the random variable $\mathcal{Z} = W\mathcal{Q}$
                follows the distribution \eqref{miche} with $\beta = \sum_{j=1}^k \beta_{i_j}$.
            \end{proposition}
            \begin{proof}
                The proof is immediate considering that $\mathcal{Q}$ comes from the sum of i.i.d.\ positive random variables each one with Laplace transform given by \eqref{lap} and then divided by $W$. Therefore 
                $\mathcal{Z} = W \mathcal{Q}$ has distribution given by \eqref{miche} with $\beta = \sum_{j=1}^k \beta_{i_j}$.
            \end{proof}
            An immediate corollary of this result is
            \begin{corollary}[Marginal distribution]
                Consider the distribution defined in equation \eqref{integralformula}. Then its marginal on $Q_i$ is given by
                \begin{align}
                    f_{Q_i} (q_i) = {} & q_i^{\nu \beta_i -1}(1 -q_i)^{\nu (\bar{\beta} - \beta_i) - 1} \label{marginal} \\
                    & \times \int_{0}^\infty y^{\nu \bar{\beta} -1} E^{\beta_i}_{\nu, \nu \beta_i} (-(yq_i)^\nu) E^{\bar{\beta}-\beta_i}_{\nu,\nu(\bar{\beta}-\beta_i)} ((-y(1-q_i)^\nu)) \, \mathrm{d}y. \notag
                \end{align}
            \end{corollary}
            As the three-parameter Mittag-Leffler function has a representation as an $H$-function  \cite{mathai},
            \begin{align}
                H_{p,r}^{m,n}(z)  = H_{p,r}^{m,n}  \!\left[ z \left| \begin{array}{l}
                ( a_i, A_i )_{1,p} \\
                ( b_i, B_i)_{1,r}  \end{array} \right. \right], \label{HF}
            \end{align}
            for  suitable choices of $(a_i, A_i)$ and $(b_i, B_i)$, the marginal distribution \eqref{marginal} can be expressed in terms of an $H$-function too.
            \begin{proposition}
                If $Q_i$ is the random variable with pdf \eqref{marginal}, then
                \begin{align}
                    & f_{Q_i} (q_i) = \frac{1}{\nu \Gamma(\beta_i)\Gamma(
                    \bar{\beta}-\beta_i)} q_i^{\nu \beta_i -1}(1 -q_i)^{-(\nu\beta_i + 1)}  \label{marginal1} \\
                    & \times \lim_{\varepsilon \downarrow 0}  H_{3,3}^{2,2} \!\left[ \left(\frac{q_i}{1-q_i}\right)^{\nu} \left| \begin{matrix} 
                    ( 1-\beta_i , 1 ) & (1-\bar{\beta}+\varepsilon,1) & (\nu(\varepsilon-\beta_i), \nu) \\
                    ( 0 ,1 ) & ( \varepsilon-\beta_i , 1 ) & (1 - \nu \beta_i , \nu)  \end{matrix} \right. \right]. \notag
                \end{align}
            \end{proposition}
            \begin{proof}
                In the integral \eqref{marginal}
                set 
                \begin{align}
                    \delta_1=\beta_i, \,\, \delta_2=\bar{\beta}-\beta_i, \,\, \bar{q}_1=q_i^{\nu}, \,\,\, \hbox{and} \,\,\, \bar{q}_2 = (1-q_i)^{\nu}.
                    \label{par}
                \end{align} 
                Denote with $I^{\nu \bar{\beta}}_{\delta_1, \delta_2}$ the resulting integral and observe that
                \begin{align}
                    I^{\nu \bar{\beta}}_{\delta_1, \delta_2}  = \frac{1}{\nu}  
                    \int_{0}^\infty y^{ \delta_1+\delta_2 -1} E^{\delta_1}_{\nu, \nu \delta_1} (- y \bar{q}_1) E^{\delta_2}_{\nu,\nu \delta_2} ( -y \bar{q}_2) \, \mathrm{d}y. \label{marginal2}
                \end{align}
                For $\delta > 0$ and $\nu \in (0,1]$ we have 
    		    \begin{align}
        			E_{\nu,\nu \delta}^\delta(z) =\frac{1}{\Gamma(\delta)} H_{1,2}^{1,1} \!\left[ -z \left| \begin{array}{l}
                    ( 1-\delta , 1 ) \\
                    ( 0 ,1 ) \,\,\, ( 1-\nu \delta , \nu )  \end{array} \right. \right]. \label{MLH}
    		    \end{align}
                For $\eta \in (0,\bar{\beta}),$ 
                by using \eqref{MLH} and Theorem 2.9 in \cite{Kilbas}, we have 
    	 	    \begin{align}
    	            I^{\eta}_{\delta_1,\delta_2}  & = \int_{0}^\infty y^{\eta -1} E^{\delta_1}_{\nu, \nu \delta_1} (- y\bar{q}_1) E^{\delta_2}_{\nu,\nu \delta_2} (- y\bar{q}_2) \, \mathrm{d}y \label{MLH2} \\ & =\frac{\bar{q}_2^{-\eta}}{\Gamma(\delta_1)\Gamma(\delta_2)} H_{3,3}^{2,2} \!\left[ \frac{\bar{q}_1}{\bar{q}_2} \left| \begin{matrix} 
                    ( 1-\delta_1 , 1 ) & (1-\eta,1) & (\nu(\delta_2 - \eta), \nu) \\
                    ( 0 ,1 ) & ( \delta_2 - \eta , 1 ) & (1 - \nu \delta_1 , \nu)  \end{matrix} \right. \right]. \notag
    		    \end{align}
                Set $\eta = \bar{\beta}-\varepsilon$ in \eqref{MLH2} with $\varepsilon \in (0, \bar{\beta}),$  and use  
                \eqref{par} to recover $
                I^{\bar{\beta}-\varepsilon}_{\beta_i,\bar{\beta}-\beta_i}.$  If $\varepsilon$ is sufficiently small, the poles $-l, \beta_i - \varepsilon - l, (\nu \beta_i - 1 -l)/\nu,l=0,1,2,\ldots$, do not coincide with the poles 
                $\beta_i + k, \bar{\beta} - \varepsilon + k, (\nu(\varepsilon - \beta_i) +k)/\nu, k=0,1,2,\ldots \, .$ Then, according to Theorem 1.1 in \cite{Kilbas}, 
                the $H$-function in \eqref{marginal1}
                makes sense for all
                $q_i \in (0,1)$ as
                $A_1 + A_2 - A_3 + B_1 + B_2 - B_ 3 = 4 - 2\nu >0.$ The claim follows 
                by taking the limit as $\varepsilon \downarrow 0$ of $
                I^{\bar{\beta}-\varepsilon}_{\beta_i,\bar{\beta}-\beta_i}.$ 
    		\end{proof}
            \begin{remark}
                By using Properties 2.1, 2.3 and 2.5 of \cite{Kilbas}, the $H$-function in \eqref{marginal1} can be rewritten interchanging  $\beta_i$ with $\bar{\beta}-\beta_i$ and 
                $q_i$ with $1-q_i$, which corresponds to commuting 
                the two Mittag-Leffler functions in \eqref{MLH2}.
            \end{remark}
            According to Theorem 1.3 and 1.4 
            \cite{Kilbas}, since 
            \begin{align}\label{conditions}
                &\sum_{i=1}^3 (B_i - A_i) = 0, \qquad
                \prod_{i=1}^3   \frac{B_i^{B_i}}{A_i^{A_i}}  =  1,  \\
                & \sum_{j=1}^3 (b_i - a_i) =  \bar{\beta} (1 - \nu) + \nu (\bar{\beta} - \varepsilon) > 0, \notag
            \end{align}
            the $H$-function in \eqref{marginal1} has a power series expansion. The following propositions rely on this property.
            \begin{proposition}
                For $q_i < 1/2$ and $\beta_i$ not a positive integer
                \begin{align}
                    & f_{Q_i} (q_i) = \frac{q_i^{\nu \beta_i -1}(1 -q_i)^{-(\nu\beta_i + 1)}}{\nu \Gamma(\beta_i)\Gamma(
                    \bar{\beta}-\beta_i)}  \notag \\ 
                    & \times \left[
                    \Gamma( \bar{\beta})  \frac{ \Gamma(-\beta_i) \Gamma(\beta_i)}{\Gamma(-\nu\beta_i) \Gamma(\nu\beta_i)}
                    + \sum_{k = 1}^{\infty} (-1)^k D_k \left(\frac{q_i}{1-q_i} \right)^{\nu k} \right], \label{expansion} 
                \end{align}
                where
                \begin{align}
                    D_k & =  
                    \frac{\Gamma( \bar{\beta}+k)\Gamma(-\beta_i - k) \, \Gamma(\beta_i + k)}{k!\, \Gamma(-\nu (\beta_i + k)) \, \Gamma(\nu(\beta_i + k))} +
                    \frac{(1-q_i)^{\nu \beta_i}\Gamma(\beta_i - k) \, \Gamma(\bar{\beta} - \beta_i + k)}{q_i^{\nu \beta_i} \Gamma(- \nu k) \, \Gamma(\nu k)}. \label{D_K}
                \end{align}
            \end{proposition}
            \begin{proof}
                Consider the $H$-function 
                $H^{2,2}_{3,3}$  in  \eqref{marginal1}. 
                If  $\beta_i$ is not a positive integer, we have $B_1 (b_2 + l) \ne B_2 (b_1+k)$ for $l,k=0,1,2,\ldots.$ Thus, thanks to \eqref{conditions}, from Theorem 1.3 of \cite{Kilbas},
                $H^{2,2}_{3,3}$  is an analytical function in $q_i^{\nu} / (1-q_i)^{\nu}$ and has the following power series expansion for $q_i< 1/2$:
                \begin{align}
                    & \sum_{k = 0}^{\infty} \frac{\Gamma(b_2 - k) \, \Gamma(1 - a_1 + k) \, \Gamma(1 - a_2 + k)}{\Gamma(1-b_3 + \nu k) \, \Gamma(a_3 - \nu k)} \frac{(-1)^k}{k!} \left(\frac{q_i}{1-q_i} \right)^{\nu k} \! \! +\left( \frac{q_i}{1-q_i} \right)^{\nu b_2} \notag \\
                    & + \sum_{k = 0}^{\infty} \frac{\Gamma(-b_2 - k)\, \Gamma(1 - a_1 + b_2 + k ) \, \Gamma(1 - a_2 + b_2 + k)}{\Gamma(1-b_3 + \nu(b_2+k))\, \Gamma(a_3 - \nu (b_2+k))} \frac{(-1)^k}{k!} \left( \frac{q_i}{1-q_i} \right)^{\nu k}.
                    \label{exp1bis}
                \end{align}
                The claim follows replacing $a_1 = 1 - \beta_i, a_2 = 1 - \bar{\beta} + \varepsilon, a_3 = \nu (\varepsilon-\beta_i), b_2 = \varepsilon-\beta_i$ and $b_3 = 1 - \nu \beta_i,$ in \eqref{exp1bis} and taking the limit as $\varepsilon \downarrow 0.$  
            \end{proof}
            \begin{remark} \label{2.5}
                If $\nu (\beta_i +k)$ is not a positive integer for $k=0,1,2,\ldots$ thanks to the reflection formula for the gamma function \cite{Kilbas}, we might simplify the expansion in \eqref{expansion} using 
                \begin{align}
                    \frac{\Gamma(-\beta_i - k) \, \Gamma(\beta_i + k)}{\Gamma(-\nu (\beta_i + k)) \, \Gamma(\nu(\beta_i + k))} = \nu \frac{\sin(\pi \nu (\beta_i+k))}{
                    \sin(\pi (\beta_i +k))}.
                    \label{refl}
                \end{align} 
                Similarly we get $\Gamma(-\nu k) \, \Gamma(\nu k) = - \pi / (k \nu \sin(\pi \nu k))$ if $\nu k$ is not a positive integer. 
            \end{remark}
            \begin{proposition}
                For $q_i > 1/2$ and $\bar{\beta}-\beta_i$ not a positive integer
                \begin{align}
                    & f_{Q_i} (q_i) = \frac{q_i^{-(\nu(
                    \bar{\beta}-\beta_i) +1)}(1 -q_i)^{\nu(\bar{\beta}-\beta_i)-1}}{\nu \Gamma(\beta_i)\Gamma(
                    \bar{\beta}-\beta_i)} \notag \\
                    & \times \left[
                    \Gamma( \bar{\beta})  \frac{ \Gamma(-(\bar{\beta}-\beta_i)) \Gamma(\bar{\beta}-\beta_i)}{\Gamma(-\nu(\bar{\beta}-\beta_i)) \Gamma(\nu(\bar{\beta}-\beta_i))}
                    + \sum_{k = 1}^{\infty} (-1)^k D_k \left(\frac{1-q_i}{q_i} \right)^{\nu k} \right] \label{expansion2}
                \end{align}
                where
                \begin{align}
                    D_k =  &\frac{\Gamma( \bar{\beta}+k)}{k!} 
                    \frac{ \Gamma(-(\bar{\beta}-\beta_i+k )) \Gamma(\bar{\beta}-\beta_i +k )}{\Gamma(-\nu(\bar{\beta}-\beta_i +k )) \Gamma(\nu(\bar{\beta}-\beta_i + k))}\!  
                    \\
                    & + \left(\frac{q_i}{1-q_i} \right)^{\nu(\bar{\beta} - \beta_i)} \frac{\Gamma(\beta_i + k) \, \Gamma(\bar{\beta} -  \beta_i - k)} {\Gamma(- \nu k) \, \Gamma(\nu k)}. \notag
                \end{align}
            \end{proposition}
            \begin{proof}
                Consider again the $H$-function 
                $H^{2,2}_{3,3}$  in  \eqref{marginal1}.
                If $\bar{\beta}-\beta_i$ is not a positive integer, we have $A_1 (1-a_2 + l) \ne A_2 (1-a_1+k)$ for $l,k=0,1,2,\ldots$ 
                From \eqref{conditions} and Theorem 1.4 of \cite{Kilbas},
                $H^{2,2}_{3,3}$  is an analytical function in $q_i^{\nu} / (1-q_i)^{\nu}$ and has the following power series expansion for $q_i > 1/2$:
                \begin{align}
                    &\!\! \sum_{k = 0}^{\infty} \frac{\Gamma(1-a_1 + k) \, \Gamma(b_2 + 1 - a_1 +k) \, \Gamma(a_1 - a_2 - k)}{\Gamma(a_3 + \nu(1 - a_1 + k)) \, \Gamma(1-b_3 - \nu(1-a_1+k))} \frac{(-1)^k}{k!} \left[\frac{q_i}{1-q_i} \right]^{\!\nu (a_1-1-k)} \notag \\
                    & \!\!\!\! + \sum_{k = 0}^{\infty} \! \frac{\Gamma(1-a_2 + k) \, \Gamma(b_2 + 1 - a_2 +k) \, \Gamma(a_2 - a_1 - k)}{\Gamma(a_3 + \nu(1 - a_2 + k)) \, \Gamma(1-b_3 - \nu(1-a_2+k))} \frac{(-1)^k}{k!}\!\! \left[\frac{q_i}{1-q_i} \right]^{\!\nu (a_2-1-k)}
                    \label{exp1}
                \end{align}
                The claim follows replacing $a_1 = 1 - \beta_i, a_2 = 1 - \bar{\beta} + \varepsilon, a_3 = \nu (\varepsilon-\beta_i), b_2 = \varepsilon-\beta_i$ and $b_3 = 1 - \nu \beta_i,$ in \eqref{exp1} and taking the limit as $\varepsilon \downarrow 0.$ 
            \end{proof}
                By using the reflection formula for the gamma function, also the expansion  \eqref{expansion2} might be simplified similarly to what has been addressed in Remark \ref{2.5}.

    \section{An alternative generalization}

        \noindent    We now give an alternative generalization
        with desirable properties which in addition approximates the fractional Dirichlet distribution with density \eqref{integralformula} for appropriate values of the parameters. 
    
        Let us thus consider a random vector $\bm{Q} = (Q_1,\dots,Q_{n-1})$, $n \ge 2$, with the following probability density function:  
	    \begin{align}\label{closedform}
			f_{\bm{Q}}(q_1,\dots,q_{n-1}) = {} &\left( \prod_{i=1}^{n-1} q_i^{\nu \beta_i-1} \right)
			\left( 1-\sum_{i=1}^{n-1} q_i\right)^{\nu \beta_n-1} \\
			& \times \frac{\nu^{n-1}  \Gamma(\bar{\beta})}{\Gamma(\beta_1) \cdot \ldots \cdot \Gamma(\beta_n)} \biggl(q^{\nu}_1+\cdots+ \Bigl(1- \sum_{i=1}^{n-1}q_i\Bigr)^{\nu}\biggr)^{-\bar{\beta}}, \notag
		\end{align}
		for $q_1, \ldots, q_{n-1} \in (0,1)$, $q_1+\dots+q_{n-1}<1$, $\nu >0$, $\beta_i >0$, $i=1,\dots n$, $\bar{\beta} = \beta_1+\dots\beta_n$.
		
		For the sake of clarity we check that $f_{\bm{Q}}(q_1,\dots,q_{n-1})$ as given in (\ref{closedform}) is a genuine probability density function. This will follow by proving that \begin{align}
		    \frac{\nu^{n-1}  \Gamma(\bar{\beta})}{\Gamma(\beta_1) \cdot \ldots \cdot \Gamma(\beta_n)}
		\end{align} in the rhs of (\ref{closedform}) plays the role of a normalization coefficient.

        \begin{theorem}
            We have
            \begin{align} \label{(eqintegrale)}
                \int_0^{1} {\rm d}q_1 \cdots & \int_0^{1-q_1-\cdots-q_{n-2}} \!\!\!\!\!\!\!\!\!\! {\rm d}q_{n-1} \left( \prod_{i=1}^{n-1} q_i^{\nu \beta_i-1} \right)  \left( 1-\sum_{i=1}^{n-1} q_i\right)^{\nu \beta_n-1} \\
			    & \times \biggl(q^{\nu}_1+\cdots+ \Bigl(1- \sum_{i=1}^{n-1}q_i\Bigr)^{\nu}\biggr)^{-\bar{\beta}} =  \frac{\Gamma(\beta_1) \cdot \ldots \cdot \Gamma(\beta_n)}{\nu^{n-1}  \Gamma(\bar{\beta})}. \notag
            \end{align}
        \end{theorem}
        \begin{proof}
            Observe that the lhs of \eqref{(eqintegrale)} can be rewritten as
            \begin{align}\label{(eqintegrale1)}
                I = \int_0^{1} & {\rm d}q_1 \cdots \int_0^{1-q_1-\cdots-q_{n-2}} \!\!\!\!\!\! {\rm d}q_{n-1} \left( \prod_{i=1}^{n-1} q_i^{-1} \right) \\ & \times (1-\bar{q})^{-1} \prod_{i=1}^{n-1} \left( \frac{q_i}{1-\bar{q}} \right)^{\nu \beta_i}
                \left( 1+ \sum_{i=1}^{n-1} \left(\frac{q_i}{1-\bar{q}}\right)^{\nu} \right)^{-\bar{\beta}}, \notag
            \end{align}
            where $\bar{q}=q_1+ \cdots + q_{n-1}$.
            Apply the change of variables $(1-\bar{q})/q_i = z_i$ for $i=1, \ldots, n-1$, in multivariate integration.
            Thus, we have 
            \begin{align}\label{transformation1}
                & q_i = \frac{\prod_{j \in {\mathcal I}_{n-1,i}}  z_j}{\prod_{j=1}^{n-1} z_j + \sum_{k=1}^{n-1} \prod_{j \in {\mathcal I}_{n-1,k}}  z_j}, \qquad i = 1, \ldots, n-1,  \\
                &J= \frac{\prod_{j=1}^{n-1} z_j^{n-2}}{\left(\prod_{j=1}^{n-1} z_j + \sum_{k=1}^{n-1} \prod_{j \in {\mathcal I}_{n-1,k}}  z_j\right)^n}, \notag
            \end{align}
            where ${\mathcal I}_{n-1,k} = \{1, \ldots, k-1, k+1, \ldots, n-1\}$ for $k=1, \ldots, n-1$ and $J$ is the Jacobian of the transformation. Note that 
            \begin{align}
                1-\bar{q} = \frac{\prod_{j=1}^{n-1} z_j}{\left(\prod_{j=1}^{n-1} z_j + \sum_{k=1}^{n-1} \prod_{j \in{\mathcal I}_{n-1,k}}  z_j\right)}.
                \label{transformation2}
            \end{align}
            By putting \eqref{transformation1} and \eqref{transformation2} in \eqref{(eqintegrale1)} we have
            \begin{align}
                I = \int_0^{\infty} {\rm d}z_1 \cdots \int_0^{\infty} {\rm d}z_{n-1} \prod_{i=1}^{n-1} z_j^{- \nu \beta_j -1 } \left( 1 + \sum_{j=1}^{n-1} \frac{1}{z_j^{\nu}} \right)^{- \bar{\beta}}.
                \label{transformation3}
            \end{align}
            Apply  the  change  of  variables $z_i^{\nu} = t_i$ for $i=  1,...,n-1$, in multivariate integration. Then, we have $z_i = t_i^{1/\nu}$ for $i=  1,...,n-1$,
            and $\nu^{1-n}\prod_{i=1}^{n-1} t_i^{1/\nu-1}$ is the Jacobian of this transformation. From 
            \eqref{transformation3}
            \begin{align}
                I = \frac{1}{\nu^{n-1}} \int_0^{\infty} {\rm d}t_1 \cdots \int_0^{\infty} {\rm d}t_{n-2}  \: I_{n-2}(t_1, \ldots, t_{n-2}),
                \label{transformation4}
            \end{align}
            where
            \begin{align}
                I_{n-2}(t_1, \ldots, t_{n-2}) = \int_0^{\infty} {\rm d}t_{n-1}
                \prod_{i=1}^{n-1} t_j^{- \beta_j -1 } \left( 1 + \sum_{j=1}^{n-1} \frac{1}{t_j} \right)^{- \bar{\beta}} .
                \label{transformation4bis}
            \end{align}
            Observe that $I_{n-2}(t_1, \ldots, t_{n-2})$ in \eqref{transformation4bis}  can be rewritten as
            \begin{align}\label{transformation5} 
                I_{n-2}&(t_1, \ldots, t_{n-2}) \\
                = {} & \int_0^{\infty} {\rm d}t_{n-1}
                \prod_{j=1}^{n-1} t_j^{ \bar{\beta}- \beta_j -1 } \left( \prod_{j=1}^{n-1} t_j + \sum_{k=1}^{n-1} \prod_{j \in {\mathcal I}_{n-1,k}} t_j \right)^{-\bar{\beta}} \notag\\
                = {} &\prod_{i=1}^{n-2} t_i^{-\beta_i-1} \notag \\
                & \times \int_0^{\infty} {\rm d}t_{n-1} \left( 1 + t_{n-1} \frac{\prod_{i=1}^{n-2} t_i + \sum_{k=1}^{n-2} \prod_{i \in {\mathcal I}_{n-2,k}} t_i}{\prod_{i=1}^{n-2} t_i} \right)^{-\bar{\beta}} t_{n-1}^{\bar{\beta} - \beta_{n-1}-1}. \notag
            \end{align}
            With the change of variable 
            $z = t_{n-1} (  \prod_{i=1}^{n-2} t_i + \sum_{k=1}^{n-2} \prod_{i \in {\mathcal I}_{n-2,k}} t_i)/
            \prod_{i=1}^{n-2} t_i$ and by recalling the Mellin trasform of $(1+z)^{-\bar{\beta}}$, we recover
            \begin{align}\label{transformation7}
                & I_{n-2}(t_1, \ldots, t_{n-2}) \\
                & = \frac{
                \prod_{i=1}^{n-2} t_i^{\bar{\beta}-\beta_i-\beta_{n-1} -1}}{(\prod_{i=1}^{n-2} t_i + \sum_{k=1}^{n-2} \prod_{i \in {\mathcal I}_{n-2,k}} t_i)^{\bar{\beta}-\beta_{n-1}}} \frac{\Gamma(\bar{\beta} - \beta_{n-1}) \Gamma(\beta_{n-1}) }{\Gamma(\bar{\beta})}. \notag
            \end{align} 
            Now, replace $I_{n-2}(t_1, \ldots, t_{n-2})$ in \eqref{transformation4} with the closed 
            form 
            \eqref{transformation7}. This leads us to 
            \begin{align}
                I = \frac{1}{\nu^{n-1}}\frac{\Gamma(\bar{\beta} - \beta_{n-1}) \Gamma(\beta_{n-1}) }{\Gamma(\bar{\beta})} \int_0^{\infty} {\rm d}t_1 \cdots \int_0^{\infty}  {\rm d}t_{n-3} I_{n-3}(t_1, \ldots, t_{n-3}) ,
                \label{transformation8}
            \end{align}
            where
            \begin{align}\label{transformation9}
                &I_{n-3}(t_1, \ldots, t_{n-3}) \\
                & = \int_0^{\infty} {\rm d}t_{n-2}
                \prod_{j=1}^{n-2} t_j^{ \bar{\beta} -\beta_{n-1} - \beta_j  -1 } \left( \prod_{j=1}^{n-2} t_j + \sum_{k=1}^{n-2} \prod_{j \in {\mathcal I}_{n-2,k}} t_j \right)^{-(\bar{\beta}-\beta_{n-1}) }. \notag
            \end{align}
            By comparing the integral in \eqref{transformation9} with that in \eqref{transformation5}, we observe that the former has the same expression of the latter with $\bar{\beta}$ replaced by $\bar{\beta} - \beta_{n-1}.$ Thus, by recurring to the same arguments employed to compute
            $I_{n-2}(t_1, \ldots, t_{n-2})$ we recover
            \begin{align} \label{transformation10}
                I_{n-3}(t_1, \ldots, t_{n-3}) = {} & \frac{
                \prod_{i=1}^{n-3} t_i^{\bar{\beta}-\beta_{n-1} -\beta_{n-2} 
                -\beta_i -1}}{(\prod_{i=1}^{n-3} t_i + \sum_{k=1}^{n-3} \prod_{i \in {\mathcal I}_{n-3,k}} t_i)^{\bar{\beta}-\beta_{n-1} - \beta_{n-2}}} \\
                & \times \frac{\Gamma(\bar{\beta} - \beta_{n-1} - \beta_{n-2}) \Gamma(\beta_{n-2}) }{\Gamma(\bar{\beta} - \beta_{n-1})}. \notag
            \end{align} 
            Replacing $I_{n-3}(t_1, \ldots, t_{n-3})$ in \eqref{transformation8} with the closed form 
            \eqref{transformation10} we get 
            \begin{align} \label{transformation11}
                I = {} & \frac{1}{\nu^{n-1}}\frac{\Gamma(\bar{\beta} - \beta_{n-1} - \beta_{n-2}) \Gamma(\beta_{n-1}) \Gamma(\beta_{n-2}) }{\Gamma(\bar{\beta})} \\ &\times \int_0^{\infty} {\rm d}t_1 \cdots \int_0^{\infty}   {\rm d}t_{n-4} I_{n-4}(t_1, \ldots, t_{n-4}), \notag
            \end{align}
            where
            \begin{align} \label{transformation12}
                I_{n-4}(t_1, \ldots, t_{n-4}) = {} & \int_0^{\infty} {\rm d}t_{n-3}
                \prod_{j=1}^{n-3} t_j^{ \bar{\beta} -\beta_{n-1} - \beta_{n-2} - \beta_j  -1 } \\
                & \times \left( \prod_{j=1}^{n-3} t_j + \sum_{k=1}^{n-3} \prod_{j \in {\mathcal I}_{n-3,k}} t_j \right)^{-(\bar{\beta}-\beta_{n-1} - \beta_{n-2}) }, \notag
            \end{align}
            which indeed has the same expression of $I_{n-2}$ and $I_{n-3}$ with  suitable updates of $\bar{\beta}$.
            The result follows by iterating from $i=4$ up to $i=n-1$ the computation of
            \begin{align}
                I = {} & \frac{1}{\nu^{n-1}}
                \frac{\Gamma(\bar{\beta} - \sum_{k=n-i+2}^{n-1} \beta_k) \prod_{k=n-i+2}^{n-1} \Gamma(\beta_k) }{\Gamma(\bar{\beta})} \\ & \times \int_0^{\infty} {\rm d}t_1 \cdots \int_0^{\infty}   {\rm d}t_{n-i} I_{n-i}(t_1, \ldots, t_{n-i}) \notag
            \end{align}
            with
            \begin{align}
                I_{n-i}(t_1, \ldots, t_{n-i}) = {} & \int_0^{\infty} {\rm d}t_{n-i+1}
                \prod_{j=1}^{n-i+1} t_j^{ \bar{\beta} - \sum_{k=n-i+2}^{n-1} \beta_{k}  - \beta_j  -1 } \\
                & \times \left( \prod_{j=1}^{n-i+1} t_j + \sum_{k=1}^{n-i+1} \prod_{j \in {\mathcal I}_{n-i+1,k}} t_j \right)^{-(\bar{\beta}-\sum_{k=n-i+2}^{n-1} \beta_{k}) }. \notag
            \end{align}
            We obtain the closed form expression
            \begin{align}\label{transformation13}
                I_{n-i}(t_1, \ldots, t_{n-i}) = {} & \frac{
                \prod_{j=1}^{n-i} t_j^{\bar{\beta}-\sum_{k=n-i+1}^{n-1} \beta_k
                -\beta_j -1}}{(\prod_{i=1}^{n-i} t_i + \sum_{k=1}^{n-i} 
                \prod_{i \in {\mathcal I}_{n-i,k}} t_i)^{\bar{\beta}-\sum_{k=n-i+1}^{n-1} \beta_k}} \\
                & \times \frac{\Gamma(\bar{\beta} - \sum_{k=n-i+i}^{n-1} \beta_{k}) \Gamma(\beta_{n-i+1}) }{\Gamma(\bar{\beta} - \sum_{k=n-i+2}^{n-1} \beta_{k})} \notag
            \end{align}
            with $\sum_{k=1}^{n-i} 
            \prod_{i \in {\mathcal I}_{n-i,k}} t_i = 1$ for $i=n-1.$ The last replacement with $I_1(t_1)$ gives
            \begin{align}
                I = \frac{1}{\nu^{n-1}}\frac{\Gamma(\beta_1+\beta_n) \Gamma(\beta_2) \cdots \Gamma(\beta_{n-1}) }{\Gamma(\bar{\beta})} \int_0^{\infty} {\rm d}t_1 (1+t_1)^{-(\beta_1+\beta_n)} t_1^{\beta_n-1} 
                \label{transformation14}
            \end{align}
            from which the claimed result follows by observing that
            $\int_0^{\infty} {\rm d}t_1 t_1^{\beta_n-1} (1+t_1)^{-(\beta_1+\beta_n)} = \Gamma(\beta_n) \Gamma(\beta_1)/\Gamma(\beta_1+\beta_n).$
        \end{proof}

        \begin{remark}\label{remrem}
            Alternatively, in \eqref{transformation3} use the transformation $z_1^{-1}=t_1, \ldots, z_n^{-1}=t_n$. Then, we have (cf.\ \cite{Gradshteyn} no.\ 4.638/3, p.\ 649)
            \begin{align}
                \label{transformation4ter}
                I & = \int_0^{\infty} {\rm d}t_1 \cdots \int_0^{\infty} {\rm d}t_{n-1}\frac{ \prod_{i=1}^{n-1} t_j^{\nu \beta_j -1 }}{(1+t_1^{\nu} +\cdots + t_{n-1}^{\nu})^{\bar{\beta}}} \\
                & = \frac{\Gamma(\beta_1) \cdot \ldots \cdot \Gamma(\beta_{n-1})}{\nu^{n-1}}\frac{\Gamma(\bar{\beta}-\beta_1 - \cdots - \beta_{n-1})}{\Gamma(\bar{\beta})}, \notag
            \end{align}
            which is in agreement with  \eqref{(eqintegrale)}.
        \end{remark}
        
        \begin{remark}		
        	On the $n$-dimensional simplex $\Delta_n$ the probability density of the random vector $(Q_1,\dots,Q_n)$, where $\sum_n Q_n = 1$ a.s., writes
    		\begin{align}\label{closed2}
    			&\mathbb{P}\left( (Q_1, \dots, Q_n) \in \mathrm d (q_1, \dots, q_n) \right) \\
		    	&= \frac{\nu^{n-1}  \Gamma(\bar{\beta})}{\Gamma(\beta_1) \cdots \Gamma(\beta_n)} (q^{\nu}_1+\cdots+ q_n^\nu)^{-\bar{\beta}}
		    	\prod_{i=1}^n q_i^{\nu \beta_i-1} \notag \\
                &= \frac{\nu^{n-1}}{B(\bm{\beta})}
                \prod_{i=1}^n q_i^{-1} \left(\frac{q_i^{\nu}}{\sum_{i=1}^n q_i^\nu}\right)^{\beta_i}, \notag
		    \end{align}
		    with $B(\bm{\beta})=\prod_{i=1}^n \Gamma(\beta_i)/\Gamma(\sum_{i=1}^n \beta_i)$.
		    In short we write $(Q_1,\dots,Q_n) \sim \text{GDIR}(\nu,\bm{\beta})$.
		    
            Notice that for $\nu=1$ the Dirichlet$(\bm{\beta})$ is obtained. In this case the random vector
		    $(Q_1,\dots,Q_n)$ is uniformly distributed on $\Delta_n$ for $\beta_i=1$, $i \in \mathbb{N}^*$.	If instead we only let $\beta_i=1$,
		    \begin{align}
			    \mathbb{P}\left( (Q_1, \dots, Q_n) \in \mathrm d (q_1, \dots, q_n) \right)
                = \frac{\nu^{n-1}  (n-1)!}{(q^{\nu}_1+\cdots+ q_n^\nu)^{n}}
	    		\prod_{i=1}^n q_i^{\nu -1}
    		\end{align}
		    which is symmetric but clearly not uniform.
		    If $\beta_i=1/\nu$ (again symmetric) we obtain
		    \begin{align}
			    \mathbb{P}\left( (Q_1, \dots, Q_n) \in \mathrm d (q_1, \dots, q_n) \right)
			    = \frac{\nu^{n-1}  \Gamma(n/\nu)}{\Gamma(1/\nu)^n} (q^{\nu}_1+\cdots+ q_n^\nu)^{-n/\nu}
		    \end{align}
		\end{remark}
		
		\begin{remark}
		    The alternative generalized Dirichlet distribution considered in this section (i.e.\ that with pdf \eqref{closedform}) can be derived by the same procedure described in Section \ref{sec} with $(Z_i)^\nu$ distributed as Gamma$(\beta_i,1)$, $i=1,\dots,n-1$.
		    Note that the random variable $X$ such that
		    $X^\nu$, $\nu >0$, is Gamma$(\alpha,1)$-distributed, $\alpha>0$, is a special case of the
		    \emph{generalized Gamma} distribution (see e.g.\ \cite{kotz}, Section 8.7). In particular, $X$ has pdf
		    \begin{align}
		        f_X(x) = \nu\frac{x^{\nu\alpha-1} e^{-x^\nu}}{\Gamma(\alpha)} \bm{1}_{\mathbb{R}_+},
		        \label{genGAMMA}
		    \end{align}
		    and Laplace transform (from (2.3.23) of \cite{mathai} and the definition of Wright functions)
		    \begin{align}
		        \mathbb{E} e^{-zX} = \frac{\nu z^{-\nu \alpha}}{\Gamma(\alpha)} \sum_{k=0}^\infty \frac{(-z^{-\nu})^k}{k!}\Gamma(\nu(k+\alpha)).
		    \end{align}
		\end{remark}
		\begin{remark}
            The generalized Dirichlet pdf 
	           \eqref{closedform}
	           turns out to be a reasonably good approximation of the fractional Dirichlet pdf \eqref{integralformula} for 
	            $\beta_i < 1$ (see for example 
	            Fig.~\ref{fig2}). A partial explanation is that for $\lambda=1, \beta <1,$ and $\nu \in (0,1]$ the fractional Gamma pdf \eqref{miche} has a rather similar shape to the generalized Gamma pdf \eqref{genGAMMA}, as Fig.~\ref{Fig5} shows for $\beta=0.2$ and $\beta=0.4$. For $\beta_i>1,$
	            the fractional  Dirichlet pdf exhibits a behaviour different from the generalized  Dirichlet pdf (see for example
	            Fig.~\ref{fig1}). Indeed, Fig.~\ref{Fig5} shows a different shape of the fractional Gamma pdf compared to the generalized Gamma pdf for $\beta=2$ and $\beta=3$. 
	            \begin{figure}
                    \includegraphics[scale=.7]{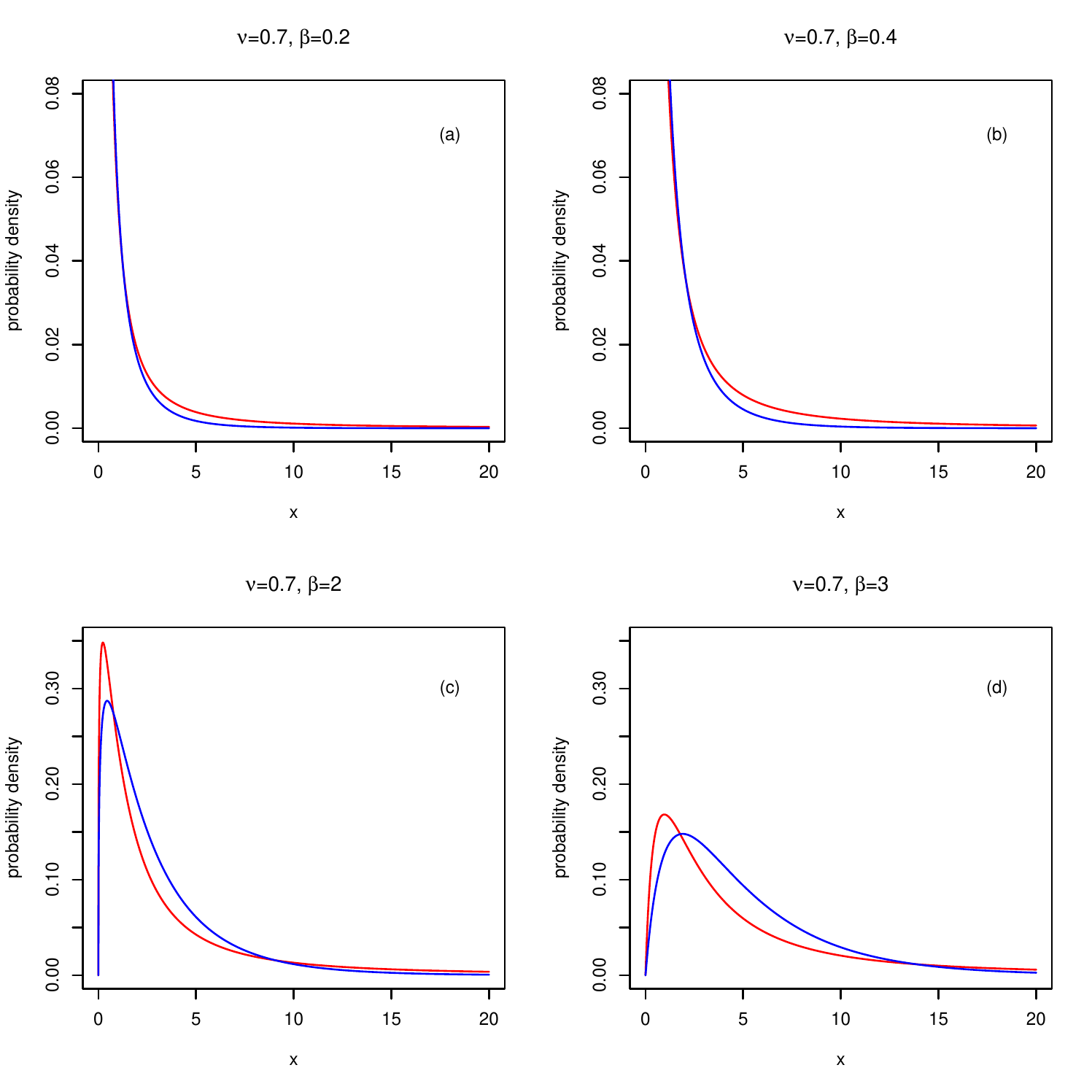}
                    \caption{\label{Fig5}Comparison of the fractional Gamma pdf \eqref{integralformula}
                    (red line) versus the generalized Gamma pdf \eqref{closedform}
                    (blue line) for $\nu=0.7$ and $\beta=0.2$ in (a) and $\beta=0.4$ in (b), as in Fig.~\ref{fig2},  $\beta=2$ in (c) and $\beta=3$ in (d), as in Fig.~\ref{fig1}. }
                \end{figure}
		    \end{remark}
		    \begin{proposition}[Conjugate distribution]
    		    The generalized Dirichlet distribution GDIR$(\nu,\bm{\beta})$ (with pdf \eqref{closed2}) is the conjugate prior to a re-parametrized Multinomial distribution with pmf
		        \begin{align}\label{lik}
		            \frac{N}{x_1!\cdot \ldots \cdot x_n!} \left( \prod_{i=1}^n q_i^{\nu x_i} \right)\left( q_1^\nu + \dots + q_n^\nu\right)^{- \sum_{i=1}^n x_i},
		        \end{align}
		        where $N \in \{1,2,\dots\}$, $x_i \in \{0, \dots, N\}$, $i = 1,\dots, n$, $n \in \{1,2,\dots\}$, $q_1 + \dots +q_n = 1$, $\nu >0$.
		        In particular, if the prior is GDIR$(\nu,\bm{\beta})$ and the likelihood is as in \eqref{lik}, then the posterior becomes GDIR$(\nu, \bm{\beta}+\bm{x})$.
		    \end{proposition}
		    \begin{proof}
		        The proof is a straightforward application of Bayes theorem. The reparametrization in \eqref{lik} is such that $p_i = q_i^\nu / (\sum_{i=1}^n q_i^\nu)$, $i=1,\dots, n$, are the event probabilities (i.e.\ $\sum_{i=1}^n p_i=1$).
		    \end{proof}

		\subsection{Representation in terms of Dirichlet random variables}
		    
		    In order to derive a meaningful representation in terms of Dirichlet random variables for the random vector $\bm{Q}$, we first recall the definitions of two related classes of random vectors (see \cite{gupta}).
		    
		    \begin{definition}[Liouville distribution of the first kind]
		        Let $\bm{X}=(X_1,\dots,X_n)$ be an absolutely continuous random vector supported on the $n$-dimensional positive orthant, i.e.\
		        $R^n =\{ (x_1,\dots,x_n) \colon x_i>0\text{ for each } i=1,\dots,n \}$. It is said to have \emph{Liouville distribution of the first kind} if its joint probability density function
		        writes
		        \begin{align}\label{liou1}
		            f_{\bm{X}} (x_1,\dots,x_n) \propto f\left( \sum_{i=1}^n x_i \right) \prod_{i=1}^n x_i^{a_i-1},
		        \end{align}
		        where $a_i>0$, $i=1,\dots,n$, and $f$ is a positive continuous function
		        satisfying $\int_{\mathbb{R}_+}y^{a-1} f(y) \mathrm{d}y < \infty$, with $a= a_1,\dots, a_n$.
		        Further, we write $\bm{X}\sim L_n^{(1)}[f(\cdot);a_1,\dots,a_n]$.
		    \end{definition}
		    
		    \begin{definition}[Liouville distribution of the second kind]
                Let $\bm{Z}=(Z_1,\dots,Z_n)$ be an absolutely continuous random vector supported on $S_n = \{ (z_1,\dots,z_n) \colon z_i>0 \text{ for each } i=1,\dots,n, \: \sum_{i=1}^n z_i < 1 \}$.
	            It is said to have \emph{Liouville distribution of the second kind} if its joint probability density function
		        writes
		        \begin{align}\label{liou2}
		            f_{\bm{Z}} (z_1,\dots,z_n) \propto g\left( \sum_{i=1}^n z_i \right) \prod_{i=1}^n z_i^{c_i-1},
		        \end{align}
		        where $c_i>0$, $i=1,\dots,n$, and $g$ is a positive continuous function
		        satisfying $\int_{\mathbb{R}_+}y^{c-1} g(y) \mathrm{d}y < \infty$, with $c= c_1,\dots, c_n$.
		        Further, we write $\bm{Z}\sim L_n^{(2)}[g(\cdot);c_1,\dots,c_n]$.
	        \end{definition}
	        
	        \begin{remark}
	            If we let $f(t) = (1+t)^{-(a+a_{n+1})}$, $t>0$, $a_{n+1}>0$,
	            in \eqref{liou1}, then $\bm{X}$ is distributed as an inverted Dirichlet.
	            
	            If, in \eqref{liou2}, we choose $g(t)= (1-t)^{c_{n+1}-1}$, $0<t<1$, $a_{n+1}>0$, we have that $\bm{Z}$ is distributed as a Dirichlet. 
	        \end{remark}
	        
	        Proposition 3.1 of \cite{gupta} tells us what is the relationship between Liouville distributions of the first and of the second kind (and hence between the Dirichlet and the inverted Dirichlet). Specifically, if $\bm{Z}\sim L_n^{(2)}[g(\cdot);c_1,\dots,c_n]$ and we consider the transformation
	        \begin{align}
	            X_i =\frac{Z_i}{1- \sum_{i=1}^n Z_i}, \qquad i=1,\dots,n,
	        \end{align}
	        then $\bm{X}\sim L_n^{(1)}[f(\cdot); c_1,\dots,c_n]$, where
	        \begin{align}\label{aa}
	            f(t) = (1+t)^{-(c+1)}g \left( \frac{t}{1+t} \right) \qquad t>0.
	        \end{align}
	        Plainly, the converse relation is true as well: inverting \eqref{aa} (letting $h=t/(1+t)$) we have
	        \begin{align}
	            g(h) = \left( \frac{1}{1-h} \right) f \left( \frac{h}{1-h} \right), \qquad  0<h<1.
	        \end{align}
	        As a simple example, considering $f(t)= (1+t)^{-(c+c_{n+1})}$, $t>0$ (inverted Dirichlet), we readily obtain
	        $g(1-h)^{c_{n+1}-1}$ (Dirichlet).
	        
	        Now, by exploiting the above definition we prove the following distributional representation for $\bm{Q}$.
	        
	        \begin{proposition}
	            Let $\bm{Q} = (Q_1, \dots, Q_{n-1})$ be distributed
	            with pdf \eqref{closedform}. Then the random vector
	            $\bm{M} = (M_1,\dots,M_{n-1})$ such that
	            \begin{align}\label{a}
	                M_i = \frac{\left( \frac{Q_i}{1- \sum_{i=1}^{n-1}Q_i} \right)^\nu}{1+ \left( \frac{Q_i}{1- \sum_{i=1}^{n-1}Q_i} \right)^\nu}, \qquad i=1, \dots, n-1,
	            \end{align}
	            is distributed as a Dirichlet$(\bm{\beta}=(\beta_1,\dots,\beta_n))$.
	            
	            Conversely, if $\bm{M}\sim\text{Dirichlet}(\bm{\beta})$ we have that $\bm{Q} = (Q_1, \dots, Q_{n-1})$ such that
	            \begin{align}\label{b}
	                Q_i = \frac{\left( \frac{M_i}{1- \sum_{i=1}^{n-1}M_i} \right)^{\frac{1}{\nu}}}{1+ \left( \frac{M_i}{1- \sum_{i=1}^{n-1}M_i} \right)^{\frac{1}{\nu}}}, \qquad i=1, \dots, n-1,
	            \end{align}
	            is distributed with pdf \eqref{closedform}.
	        \end{proposition}
	        \begin{proof}
	            Let us define the random vector $\bm{Y}=(Y_1,\dots,Y_{n-1})$ such that
	            \begin{align}
	                Y_i = \left( \frac{Q_i}{1- \sum_{i=1}^{n-1}Q_i}, \right)^\nu \qquad i=1, \dots, n-1,
	            \end{align}
	            and let $\beta^* = \beta_1 + \dots+ \beta_{n-1}$.
	            Combining the transformations in the proof of Theorem \ref{(eqintegrale)} and of Remark \ref{remrem} we see that $\bm{Y}$ has pdf
	            \begin{align}
	                f_{\bm{Y}}(y_1,\dots,y) = \frac{\Gamma(\beta^*+\beta_n)}{\Gamma(\beta_1)\cdot \ldots \cdot \Gamma(\beta_n)}\left( 1+\sum_{i=1}^{n-1} y_i \right)^{-\beta^* -\beta_n}
	                \prod_{i=1}^{n-1} y_i^{\beta_i -1},
	            \end{align}
	            and hence $\bm{Y} \sim L_{n-1}^{(1)}[(1+\cdot)^{-(\beta^*+\beta_n)};\beta_1, \dots, \beta_{n-1}]$ (i.e.\ an inverted Dirichlet distribution).
	            
	            By using the mentioned transformations between Liouville distributions of first and second kinds we have that
	            $\bm{M} \sim L_{n-1}^{(2)}[(1-\cdot)^{\beta_n-1};\beta_1, \dots, \beta_{n-1}]$ (Dirichlet), where
	            \begin{align}
	                M_i = \frac{Y_i}{1+ \sum_{i=1}^{n-1}Y_i}, \qquad i=1, \dots, n-1,
	            \end{align}
	            and \eqref{a} follows.
	            
	            Finally, a rewriting of the components of $\bm{Q}$ in terms of those of $\bm{Y}$, leads easily to \eqref{b}.
	        \end{proof}

        \section{Monte Carlo Simulations}

            \noindent The simulation of the random variables $\bm{Q}$ for the fractional Dirichlet distribution is straightforward based on the construction presented in section 2. First, one needs to generate random variables with density \eqref{miche} and one can use the mixture representation discussed in \cite{cahoypolito}
            \begin{align*}
                X_i \stackrel{d}{=} U_i^{1/\nu} V_\nu,
            \end{align*}
            where $U_i$ is $\mathrm{Gamma}(\beta_i,\lambda)$-distributed and $V_\nu$ is strictly positive-stable distributed with $\exp(-s^\nu)$ as the Laplace transform of the probability density function. Summing the $X_i$ to get $W$ and dividing $X_i$ by $W$ gives $Q_i$.
            \begin{remark}
                For $\beta=1$, there is an alternative representation \cite{kozubowski,fulger}:
                \begin{align*}
                    X \stackrel{d}{=} \Xi Z^{1/\nu},
                \end{align*}
                where $\Xi$ is $\mathrm{Exp}(\lambda)$-distributed and $Z$ is Cauchy-distributed.
            \end{remark}
            The behaviour of the fractional Dirichlet distribution in the case $N=2$ is shown in Fig.~\ref{fig1} for $\nu = 0.7$, $\beta_1 = 2$, $\beta_2 = 3$. In this case, the generalized Dirichlet distribution is not a good approximation.
            \begin{figure}
                \includegraphics[width=\textwidth]{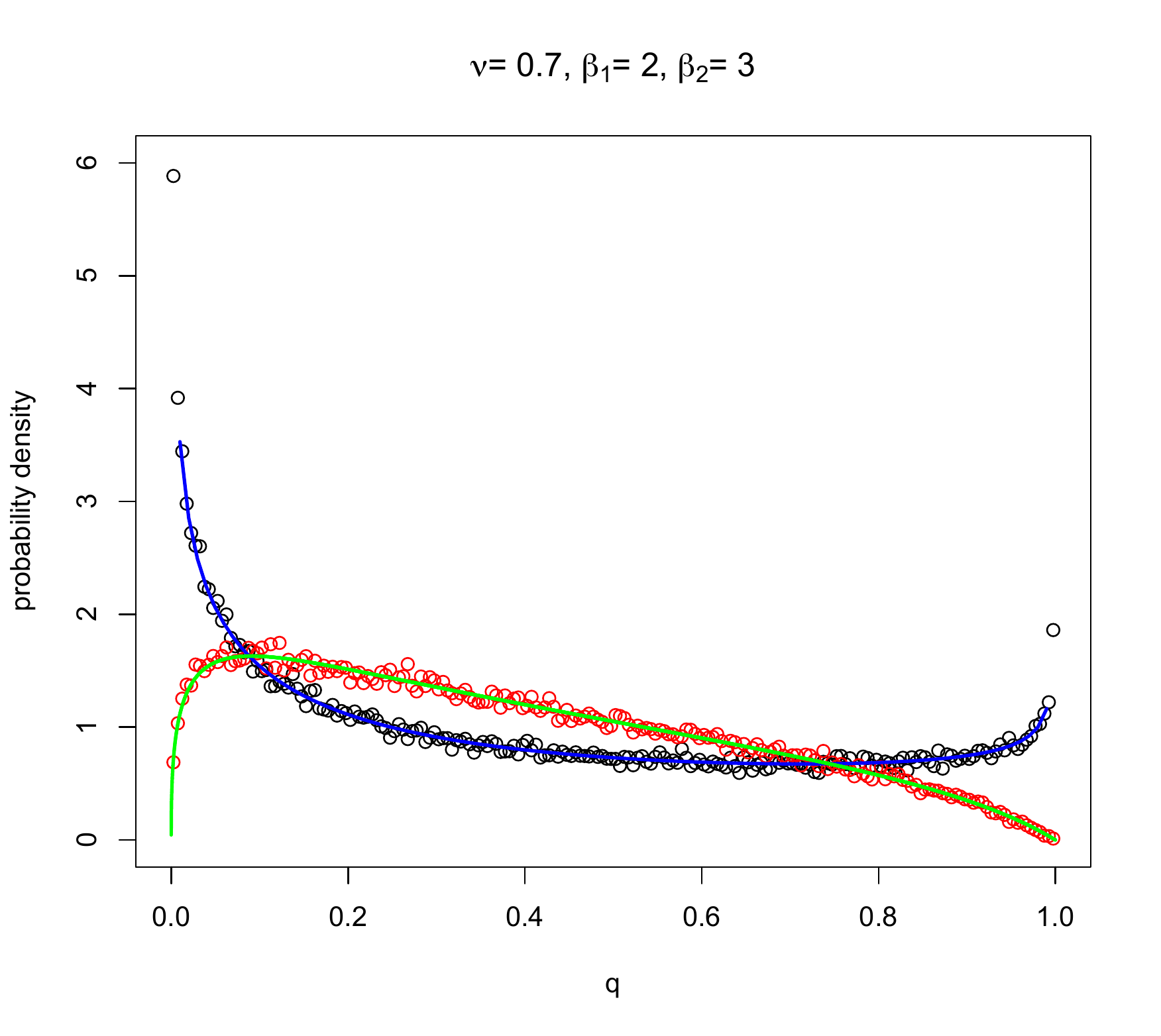}
                \caption{\label{fig1}Probability density function for $\nu=0.7$, $\beta_1=2$, $\beta_2=3$. The black circles represent the results of a Monte Carlo simulation for the fractional Dirichlet distribution and the blue line is the corresponding theoretical value from equation \eqref{integralformula}. The red circles come from a Monte Carlo simulation of the generalized Dirichlet distribution and the green curve is the plot of the theoretical probability density function.}
            \end{figure}
            This is not the case for $N=2$, $\nu=0.7$, $\beta_1=0.2$ and $\beta_2=0.4$ where the generalized Dirichlet distribution is a reasonably good approximation of the fractional Dirichlet distribution. This is represented in Fig.~\ref{fig2}.
            \begin{figure}
                \includegraphics[width=\textwidth]{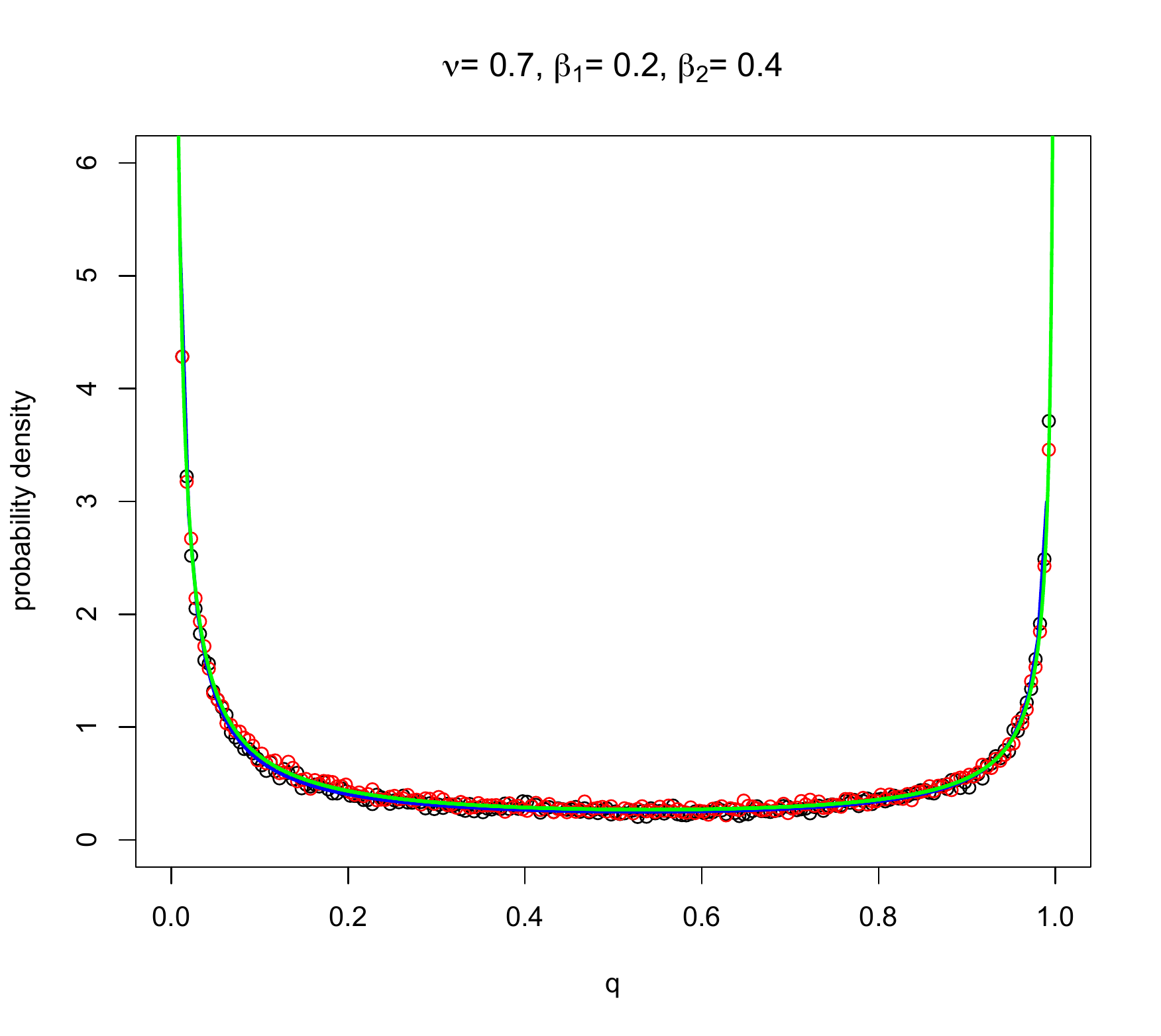}
                \caption{\label{fig2}Probability density function for $\nu=0.7$, $\beta_1=0.2$, $\beta_2=0.4$. Circles and curves have the same meaning as in Fig.~\ref{fig1}.}
            \end{figure}
            For larger values of the parameters $\beta_i$, one gets a unimodal distribution in both cases as shown in Fig.~\ref{Fig3} for $N=2$, $\nu=0.95$, $\beta_1=10$, $\beta_2=30$.
            \begin{figure}
                \includegraphics[width=\textwidth]{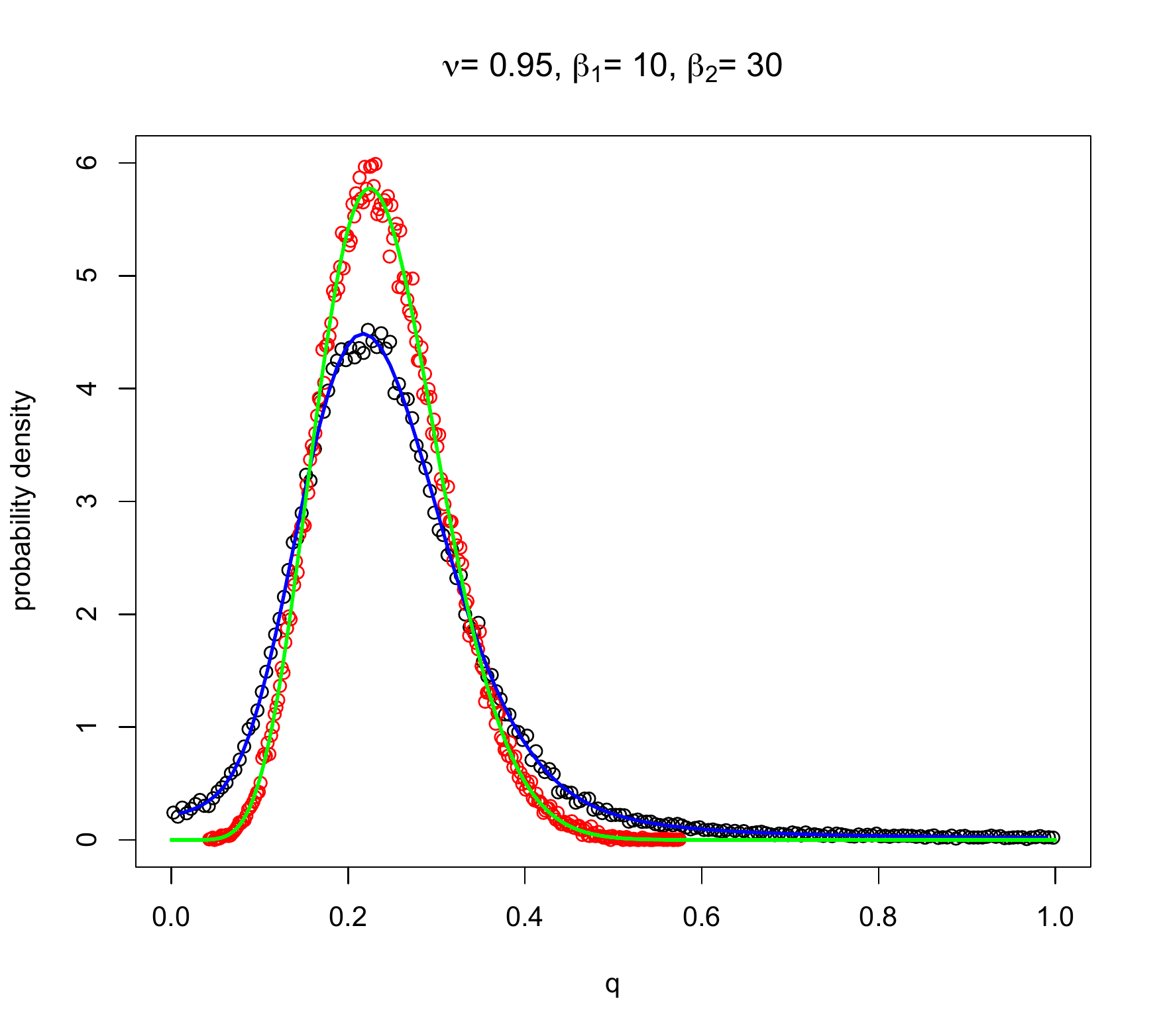}
                \caption{\label{Fig3}
                Probability density function for $\nu=0.95$, $\beta_1=10$, $\beta_2=30$. Circles and curves have the same meaning as in Fig.~\ref{fig1}.}
            \end{figure}
            In Fig.~\ref{Fig4}, the heavy character of the right tail of the generalized Dirichlet distribution is highlighted by the log-log plot.
            \begin{figure}
                \includegraphics[width=\textwidth]{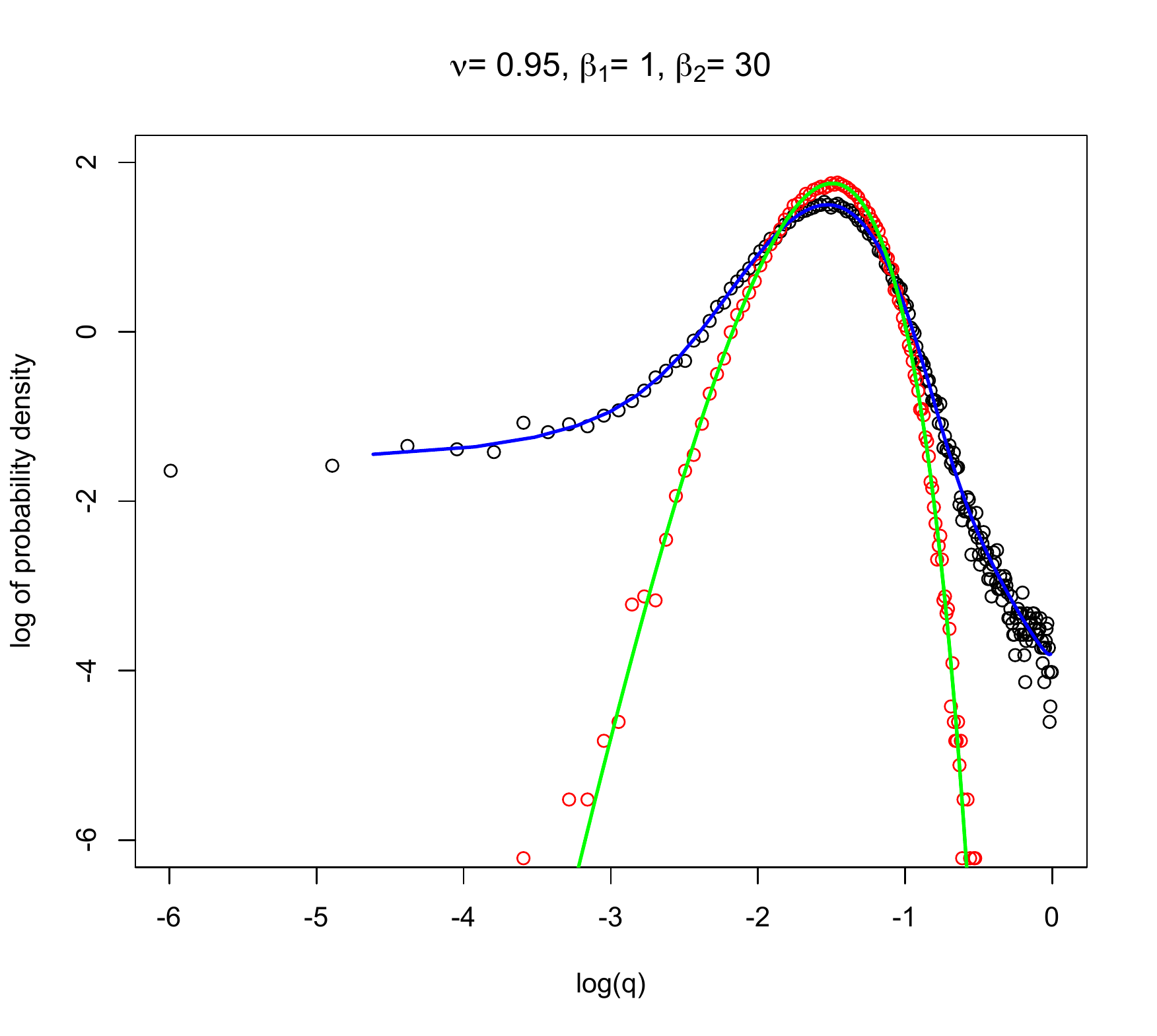}
                \caption{\label{Fig4}
                Double logarithmic plot of the probability density function for $\nu=0.95$, $\beta_1=10$, $\beta_2=30$. Circles and curves have the same meaning as in Fig.~\ref{fig1}.}
            \end{figure}

\section*{Acknowledgements}

\noindent F.~Polito has been partially supported by the project ``Memory in Evolving Graphs'' (Compagnia di San Paolo/Università degli Studi di Torino).





 \bigskip \smallskip

 \it

 \noindent
$^1$ Department of Mathematics \emph{G. Peano} \\
Università degli Studi di Torino \\
Via C. Alberto 10, 10123, Torino, Italy \\[4pt]
e-mails: elvira.dinardo@unito.it, federico.polito@unito.it\\[12pt]
$^2$ Department of Mathematics\\
University of Sussex\\
Falmer, Brighton, BN1 9QH, UK\\[4pt]
  e-mail: e.scalas@sussex.ac.uk (Corr. author)

\end{document}